\documentclass{article}
\usepackage[utf8]{inputenc}
\usepackage{color}
\usepackage{epsfig}
\usepackage{graphics}
\usepackage{epstopdf}
\usepackage{amsmath}  
\usepackage{amssymb}  
\usepackage{mathtools}
\newcommand{\R}{\mbox{$I \kern -4pt R$}}

\newtheorem{theorem}{Theorem}

\newtheorem{definition}[theorem]{Definition}

\newtheorem{lemma}[theorem]{Lemma}

\newenvironment{proof}[1][Proof]{\noindent\textbf{#1.} }{\ \rule{0.5em}{0.5em}}

\def\span{{\rm span \;}}

\begin{document}

\title{A generalized empirical interpolation method: \\ application of reduced basis techniques to data assimilation}
\author{Y. Maday$^{1, 2, 3}$ and O. Mula$^{1,4,5}$\\
{$^1$ {\small UPMC Univ Paris 06, UMR 7598, Laboratoire Jacques-Louis Lions, F-75005, Paris, France.}}\\
{$^2$ {\small Institut Universitaire de France.}}\\{$^3$ {\small Brown Univ, Division of Applied Maths, Providence, RI, USA.}}\\{$^4$ {\small CEA Saclay - DEN/DANS/DM2S/SERMA/LLPR - 91191 Gif-Sur-Yvette CEDEX - France}}\\{$^5$ {\small LRC MANON,  Lab. de Recherche Conventionn{\'e}e CEA/DEN/DANS/DM2S and UPMC-CNRS/LJLL.}}\\
{\small maday@ann.jussieu.fr} , {\small olga.mulahernandez@cea.fr}}
\date{}
\maketitle
\section{Introduction}
The representation of some physical or mechanical quantities, representing a scalar or vectorial function that depends on space, time or both, can be elaborated through at least  two --  possibly -- complementary approaches: the first one, called explicit hereafter, is based on the measurement of some instances of the quantity of interest  that consists in getting its value  at some points from which, by interpolation or extrapolation, the quantity is approximated in other points than where the measurements have been performed. The second approach, called implicit hereafter, is more elaborated. It is based on a model, constructed by expertise, that implicitly characterizes the quantity as a solution to some problem fed with input data. The model can e.g. be a parameter dependent partial differential equation, the simulation of which allows to get an approximation of the quantity of interest, and, actually,  many more outputs than the sole value of the quantity of interest. This second approach, when available, is more attractive since it allows to have a better understanding of the working behavior of the
phenomenon that is under consideration. In turn, it facilitates optimization, control or decision making.

Nevertheless for still a large number of problems, the numerical simulation of this model is indeed possible --- though far too expensive to be
performed in a reasonable enough time. The combined efforts  of numerical analysts, specialists of algorithms and computer scientists, together with the increase of the performances of the computers allow to increase every days the domains of application where numerical simulation can be used, to such an extent that it is possible now to rigorously adapt the approximation, degrade the models, degrade the simulation, or both in an intelligent way without sacrificing the quality of the approximation where it is required.

Among the various ways to reduce the problem's complexity stand approaches that use the smallness of
 the Kolmogorov $n$-width \cite{kolmo} of the manifold of all solutions considered when the parameters varies continuously in some range. This idea, combined with the Galerkin method is at the basis of the reduced basis method and the Proper Orthogonal Decomposition (POD) methods to solve parameter dependent partial differential equations. These approximation methods allow to build the solution to the model associated to some parameter as a linear combination of some precomputed solutions associated to some well chosen parameters. The precomputations can be lengthy but are performed off-line, the online computation has a very small complexity, based on the smallness of the Kolmogorov $n$-width. We refer to \cite{PRbook}\cite {prud'homme02:_reliab_real_time_solut_param}  for an introduction to these approaches.

Another possibility, rooted on the same idea, is the empirical interpolation method (EIM) that allows, from values of the quantity at some interpolating points, to build a linear combination of again preliminary fully determined quantities associated to few well chosen instances of the parameter. The linear combination is determined in such a way that it takes the same values at the interpolating points as the quantity we want to represent. This concept generalizes the classical -- e.g. polynomial or radial basis -- interpolation procedure and is recalled in the next section. The main difference is that the interpolating function are not a priori known but depend on the quantity we want to represent.

In this paper we first aim at generalizing further this EIM concept by replacing the pointwise evaluations of the quantity by more general measures, mathematically defined as linear forms defined on a superspace of the manifold of appropriate functions. We consider that this generalization, named Generalized Empirical Interpolation Method (GEIM), represents already an improvement with respect  to classical interpolation reconstructions.

Bouncing on this GEIM, we propose a coupled approach based on the domain decomposition of the computational domain into two parts: one small domain $\Omega_1$ where the 
Kolmogorov $n$-width of the manifold is not small and where the parametrized PDE will be simulated and the other subdomain $\Omega_2$ , much larger but with a small Kolmogorov $n$-width  
because for instance  the solution is driven over $\Omega_2$ by the behavior of the solution over $\Omega_1$. The idea is then to first construct (an approximation of) the solution from the measurements using the GEIM. In turn this reconstruction, up to the interface between $\Omega_1$ and $\Omega_2$, provides the necessary boundary conditions for solving the model over $\Omega_1$.

This is not the first attempt to use the small Kolmogorov width for another aim than the POD or reduced basis technique which are both based on a Galerkin approach. In \cite{maday1} e.g. the smallness of the Kolmogorov width is used to post-process a coarse finite element approximation and get an improved  accuracy.

The problems we want to address with this coupled approach, stem from, e.g., actual
industrial process or operations that work on a day-to-day basis; they can be observed with
experimental sensors that provide sound data and are able to characterize part of their
working behavior. 
We think that the numerical simulation and data mining approaches for analyzing real life systems are not
enough merged in order to (i) complement their strength and (ii) cope for their weaknesses. This paper is a contribution in this direction.

In the last section, we evoke the problem of uncertainty and noises in the acquisition of the data, since indeed, the data are most often polluted by noises. Due to this, statistical data
acquisition methods are used to filter out the source signals so that an improved knowledge is
accessible.
In many cases though, and this is more and more the case now, the data are far too
numerous to all be taken into account, most of them are thus neglected because people do
not know how to analyze them, in particular when the measures that are recorded are not
directly related to some directly understandable quantity.

\section{Generalized Empirical Interpolation Method}

 The rationale of all our approach relies on the possibility to approximately represent a given set, portion of a regular manifold (here the set of solution to some PDE), as a linear combination of very few computable elements. This is linked to the
  notion of $n$-width following Kolmogorov  \cite{kolmo}: 
\begin{definition}
Let $F$ be a subset of some Banach space ${\cal X}$ and $Y_n$ be a generic  $n$-dimensional subspace of ${\cal X}$. The angle between $F$ and $Y_n$ is
$$ E(F;Y_n) \coloneqq \sup_{x\in F}\inf_{y\in Y_n} \|x-y\|_{\cal X}.$$
The {\it Kolmogorov $n$-width} of $F$ in ${\cal X}$ is given by 

\begin{eqnarray}
d_n(F,{\cal X})&:= \inf\{E(F;Y_n) : Y_n \hbox{ a $n$-dimensional subspace of {\cal X}}\} \nonumber
\\
&=  \inf_{Y_n}\sup_{x\in F}\inf_{y\in Y_n} \|x-y\|_{\cal X}\ .
\end{eqnarray}

\end{definition}
The $n$-width of $F$ thus measures to what extent the set $F$ can be approximated by an $n$-dimensional subspace of ${\cal X}$. 

We assume from now on that $F$ and ${\cal X}$ are composed of functions defined over a domain $\Omega\subset \R^d$, where $d=1, 2,3$ and that $F$ is a compact set of ${\cal X}$.

\subsection{Recall of the Empirical Interpolation Method}

We begin by describing the construction of the empirical interpolation method (\cite{barrault04:_empir_inter_method}, 
\cite{m2an_magic}, \cite{magic}) that
allows us to define simultaneously the set of generating functions recursively chosen in $F$ together with the associated interpolation points. It is based on a
greedy selection procedure as outlined
in~\cite{nguyen04:_handb_mater_model,prud'homme02:_reliab_real_time_solut_param,veroy03:_poster_error_bound_reduc_basis}.
With ${\mathcal M}$ being some given large number, we assume that the dimension of the vectorial space spanned by $F$: $\span(F)$ is of dimension $\ge {\mathcal M}$.

The first generating function is $\varphi_1=  \arg \max_{\varphi \in F}
\|\varphi(\: \cdot \:)\|_{L^\infty(\Omega)}$, the associated
interpolation point satisfies $x_1 = \arg \max_{x \in \overline \Omega}
|\varphi_1(x)|$, we then set $q_1 = \varphi_1(\cdot)/\varphi_1(x_1)$ and $B_{11}^1 = 1$. We now
construct, by induction, the nested sets of interpolation points $\Xi_M =
\{x_1,\ldots,x_M\}, 1 \leq M \leq M_{\max},$ and the nested sets of basis
functions $\{q_1, \ldots, q_M\}$, where $M_{\max} \leq \mathcal{M}$ is some
given upper bound fixed {\em a priori}. For $M=2,\ldots,M_{\max}$, we first solve the interpolation problem: Find 

\begin{equation}
\label{1K10}
\mathcal{I}_{M-1}[\varphi(\cdot)] = \sum_{j=1}^{M-1} \alpha_{M-1,j}[\varphi] q_j \ ,
\end{equation}
such that
\begin{equation}
\label{1K10a}
\mathcal{I}_{M-1}[\varphi(\cdot)] (x_i) = \varphi(x_i), \quad i = 1, \ldots, M-1 \ ,
\end{equation}
that allows to define the $\alpha_{M-1,j}[\varphi], 1 \leq j \leq M,$ as it can be proven indeed that  the $(M-1)\times(M-1)$ matrix of running entry $q_j(x_i)$ is invertible, actually it is lower triangular with unity diagonal.

We then set
\begin{equation}
\label{eq5}
\forall \varphi \in F,\quad
\varepsilon_{M-1}(\varphi) = \|\varphi -\mathcal{I}_{M-1}[\varphi]\|_{L^\infty(\Omega)} \ ,
\end{equation}
and define
\begin{equation}
\varphi_M = \arg \max_{\varphi \in F} \varepsilon_{M-1}(\varphi) \ ,
\end{equation}
and
\begin{equation}
x_M = \arg \max_{x \in \overline \Omega} |\varphi_M(x) -\mathcal{J}_{M-1}[\varphi_M](x)| \ ,
\end{equation}
we finally set  $r_M(x) = \varphi_M(x)-\mathcal{J}_{M-1}[\varphi_M(x)]$, $q_M =
r_M/r_M(x_M)$ and $B^M_{ij} = q_j(x_i), 1 \leq i,j \leq M$.  

The
Lagrangian functions --- that can be used to build the interpolation operator
${\mathcal I}_M$ in $X_M$ $=$
$\span\{\varphi_i, 1\leq i\leq M\}$ $=$ $\span\{q_i, 1\leq i\leq M\}$
over the set of points $\Xi_M$ $=$ $\{x_i, 1\leq i\leq M\}$ --- verify  for any given $M$, ${\mathcal
I}_M[u(\: \cdot \: )]= \sum_{i=1}^M u(x_i) h^M_i(\: \cdot \:)$, where
$h^M_i(\: \cdot \:) = \sum_{j=1}^M q_j(\: \cdot \:) [B^{M}]^{-1}_{ji}$ (note indeed that $h_i^M(x_j) =\delta_{ij}$).

The error analysis of the interpolation procedure classically involves the
Lebesgue constant $\Lambda_M$ $=$ $\sup_{x\in\Omega}\sum_{i=1}^M$ $|h^M_i(x)|$.

\begin{lemma}
For any $\varphi\in F$, the interpolation error satisfies
\begin{equation}
\|\varphi-{\mathcal I}_M[\varphi]\|_{L^\infty(\Omega)}\leq (1+\Lambda_M) \inf_{\psi_M \in
X_M}\|\varphi-\psi_M\|_{L^\infty(\Omega)}.
\label{2K10}
\end{equation}
\end{lemma}
\noindent The last term in the right hand side of the above inequality is known as the best fit of $\varphi$ by elements in $X_M$.

\subsection{The generalization}

Let us assume now that we do not have access to the values of $\varphi\in F$ at points in $\Omega$ easily, but, on the contrary, that we have a dictionary of linear forms $\sigma \in \Sigma$ --- assumed to be continuous  in some sense, e.g. in  $L^2(\Omega)$ with norm 1 --- the application of which over each $\varphi\in F$ is easy. Our extension consists in defining $\tilde\varphi_1$, $\tilde\varphi_2$,\dots, $\tilde\varphi_M$ and a family of  associated linear forms $\sigma_1$, $\sigma_2$,\dots , $\sigma_M$ such that the following generalized  interpolation process (our GEIM) is well defined:
\begin{equation}
\label{eqGEIM}
{\cal J}_M[\varphi] = \sum_{j=1}^{M} \beta_j \tilde\varphi_j,\text{ such that } 
\forall i=1,\dots, M,\ \sigma_i({\cal J}_M[\varphi]) = \sigma_i(\varphi)
\end{equation}

Note that the GEIM reduces to the EIM when the dictionary is composed of dirac masses, defined in the dual space of ${\cal C}^0(\Omega)$.

As explained in the introduction, our generalization is motivated by the fact that, in practice, measurements provides outputs from function $\varphi$ that are some averages --- or some moments ---  of $\varphi$ over the actual size of the mechanical device that takes the measurement.

Among the questions raised by GEIM:
 \begin{itemize}
\item is there an optimal selection for the linear forms $\sigma_i$ within the dictionary $\Sigma$ ?
\item is there a constructive optimal selection for the functions $\tilde\varphi_i$?
\item given a set of linearly independent functions $\{\tilde\varphi_i\}_{i \in [1,M]}$ and a set of continuous linear forms $\{\sigma_i\}_{i \in [1,M]}$, does the interpolant (in the sense of (\ref{eqGEIM})) exist?
\item is the interpolant unique?
\item how does the interpolation process compares with other approximations (in particular orthogonal projections)?
\item Under what hypothesis can we expect the GEIM approximation to converge rapidly to $\varphi$?
 \end{itemize}

In what follows, we provide answers to these questions either with rigorous proofs or with numerical evidences. 

The construction of the generalized interpolation functions and linear forms is done recursively, following the same procedure as in the previous subsection, based on  a greedy approach, both for the construction of the interpolation linear forms $\tilde\varphi_i$ and the associated forms selected in the dictionary $\Sigma$:
The first interpolating function is, e.g.:
\begin{equation*}
\tilde \varphi_1 = \arg\ \underset{\varphi \in F}{\max} \Vert \varphi \Vert_{L^2(\Omega)},
\end{equation*}
the first interpolating linear form is:
\begin{equation*}
\sigma_1 = \arg\ \underset{\sigma\in  \Sigma}{\max} \vert \sigma (\varphi_1) \vert .
\end{equation*}
We then define the first basis function as: $\tilde q_1=\dfrac{\tilde \varphi_1}{\sigma_1 (\tilde \varphi_1)}$. The second interpolating function is:
\begin{equation*}
\tilde \varphi_2 = \arg\ \underset{\varphi \in F}{\max} \Vert \varphi -\sigma_1(\varphi) \tilde q_1     \Vert_{L^2(\Omega)} .
\end{equation*}
The second interpolating linear form is:
\begin{equation*}
\sigma_2 = \arg\ \underset{\sigma\in \cal L(\cal X)}{\max} \vert \sigma (\tilde\varphi_2 -\sigma_1(\tilde\varphi_2) \tilde q_1    ) \vert , 
\end{equation*}
and the second basis function is defined as: 
\begin{equation*}
\tilde q_2 =\dfrac{\tilde \varphi_2 -\sigma_1(\tilde \varphi_2) q_1}{\sigma_2(\tilde \varphi_2-\sigma_1(\tilde \varphi_2) q_1)},
\end{equation*}
and we proceed by induction: assuming that we have built  the set of interpolating functions $\{ \tilde q_1, \tilde q_2,\dots, \tilde q_{M-1} \}$ and the  set of associated interpolating linear forms $\{\sigma_1,\sigma_2,\dots,\sigma_{M-1} \}$, for $M>2$, we first solve the interpolation problem: find $\{\widetilde{\alpha_j^{M-1}}(\varphi)\}_j$ such that 
\begin{equation*}
\forall i=1,\dots,M-1,
\quad \sigma_i (\varphi)=\sum\limits_{j=1}^{M-1} \widetilde{\alpha_j^{M-1}}(\varphi) \sigma_i(\tilde q_j),\ 
\end{equation*}
and then compute:
\begin{equation*}
{\cal J}_{M-1}[\varphi]=\sum\limits_{j=1}^{M-1} \widetilde{\alpha_j^{M-1}}(\varphi) \tilde q_j
\end{equation*}
We then evaluate
\begin{equation*}
\forall \varphi \in F,\quad \varepsilon_M (\varphi) = \Vert \varphi - {\cal J}_{M-1}[\varphi] \Vert_{L^2(\Omega)},\ 
\end{equation*}
and  define:
\begin{equation*}
\tilde\varphi_{M}=  \arg\ \underset{\varphi \in F}{\max}\  \varepsilon_{M-1}(\varphi)
\end{equation*}
and: $\sigma_{M}=arg\ \underset{\sigma \in \Sigma}{sup} \vert \sigma(\tilde\varphi_{M}-{\cal J}_{M-1}[\tilde\varphi_{M}])  \vert$
The next basis function is then
\begin{equation*}
\tilde q_{M}=\dfrac{\tilde\varphi_{M}-{\cal J}_{M-1}[\tilde\varphi_{M}]}{\sigma_{M}(\tilde\varphi_{M}-{\cal J}_{M-1}[\tilde\varphi_{M}])}.
\end{equation*}
We finally define the matrix $\widetilde {B^M}$ such that $\tilde B_{ij}^M=\sigma_i (\tilde q_j)$, and set $\widetilde X_M \equiv \span\{ \tilde q_j,\ j\in [1,M] \} = \span \{ \tilde \varphi_j,\ j\in [1,M] \}$. It can be proven as in~\cite{nguyen04:_handb_mater_model,prud'homme02:_reliab_real_time_solut_param,veroy03:_poster_error_bound_reduc_basis}.

\begin{lemma}
\label{lemma3}
For any $M\le M_{\max}$, the set $\{ \tilde q_j,\ j\in [1,M] \}$ is linearly independent and $\tilde X_M$ is of dimension $M$.The matrix
 $B^M$ is lower triangular with unity diagonal (hence invertible) with other entries $\in [-1, 1]$. The generalized empirical interpolation procedure is well-posed in $L^2(\Omega)$.
\end{lemma}

%
In order to quantify the error  of the interpolation procedure, like in the standard interpolation procedure, we introduce the
Lebesgue constant in the $L^2$ norm: $\Lambda_M = \underset{\varphi\in F}{\sup} \dfrac{\Vert {\cal J}_M[\varphi] \Vert_{L^2(\Omega)}}{\Vert \varphi \Vert_{L^2(\Omega)}}$ i.e. the $L^2$--norm of ${\cal J}_M$. A similar result as in the previous subsection holds

\begin{lemma}
$\forall \varphi \in F$, the interpolation error satisfies:
\vspace{-0.2cm}
\begin{equation*}
\Vert \varphi-{\cal J}_M[\varphi] \Vert _{L^2(\Omega)} \leq (1+\Lambda_M) \underset{\psi_M \in \tilde X_M}{\inf}\Vert \varphi-\psi_M \Vert _{L^2(\Omega)} 
\end{equation*}
A (very pessimistic) upper-bound for $\Lambda_M$ is:
\begin{equation*}
\Lambda_M \leq  2^{M-1} \underset{i\in [1,M]}{\max} \Vert q_i \Vert_{L^2(\Omega)} 
\end{equation*} 
\end{lemma}
 
 \begin{proof} The first part is standard and relies on the fact that, for any $\psi\in\tilde X_N$ then ${\cal J}_M(\psi_M) = \psi_M$. It follows that 
\begin{equation*}\forall \psi_M \in \tilde X_M , \quad
\Vert \varphi-{\cal J}_M[\varphi] \Vert _{L^2(\Omega)}  = \Vert [\varphi - \psi_M] -{\cal J}_M[\varphi - \psi_M] \Vert _{L^2(\Omega)} \le (1+\Lambda_M) \Vert \varphi - \psi_M \Vert _{L^2(\Omega)}
\end{equation*} 
Let us now consider a given $\varphi \in F$ and its interpolant $\mathcal{J}_M[\varphi]=\sum\limits_{i=1}^M \widetilde{\alpha_i^M}(\varphi) \tilde q_i$ in dimension $M$. The constants $\widetilde{\alpha_i^M}(\varphi)$ come from the generalized interpolation problem: $\forall j \in [1,M],\ \sigma_j(\varphi)=\sum\limits^{j-1}_{i=1} \widetilde{\alpha_i^M}(\varphi)\sigma_j(\tilde q_i) + \widetilde{\alpha_j^M}(\varphi) j(\psi)$. We infer the recurrence relation for the constants: 
$$\forall j \in [1,M],\ \widetilde{\alpha_j^M}(\varphi) = \sigma_j(\psi) - \sum\limits^{j-1}_{i=1} \alpha_i(\psi)\sigma_j(q_i). $$ 
Based on the properties of the entries in matrix $\tilde B^M$ stated in lemma \ref{lemma3}, we can obtain, by recurrence,  an upper bound for each $ \widetilde{\alpha_j^M}(\varphi)$: $\forall j \in [1,M],\ | \widetilde{\alpha_j^M}(\varphi)| \leq \left(  2^{j-1} \right) \|\varphi\|_{L^2(\Omega)}$. Then, $\forall \varphi \in F, \ \forall M \le M_{\max}$: $\|\ \mathcal{J}_M(\varphi) \|_{L^2(\Omega)} \leq \left[ \sum\limits^{M}_{i=1} \left( 2^{j-1} \right)  \|q_i\|_{L^2(\Omega)} \right] \|\varphi\|_{L^2(\Omega)} $. Therefore: $\Lambda_M  \leq 2^{M-1} \underset{i\in [1,M]}{\max} \Vert q_i \Vert_{L^2(\Omega)}  $. Note that the norms of the rectified basis function $q_i$ verify $\Vert q_i \Vert_{L^2(\Omega)}\ge 1$ from the hypothesis done on the norm of the $\sigma_i$.

\end{proof}



%

%


%

%

%


\subsection{Numerical results}

The results that we present here to illustrate the GEIM are based on  data acquired  in silico using the finite element code Freefem \cite{freefem} on the domain represented on figure 1.

\begin{figure}[htbp]\center
  \includegraphics[width=5cm]{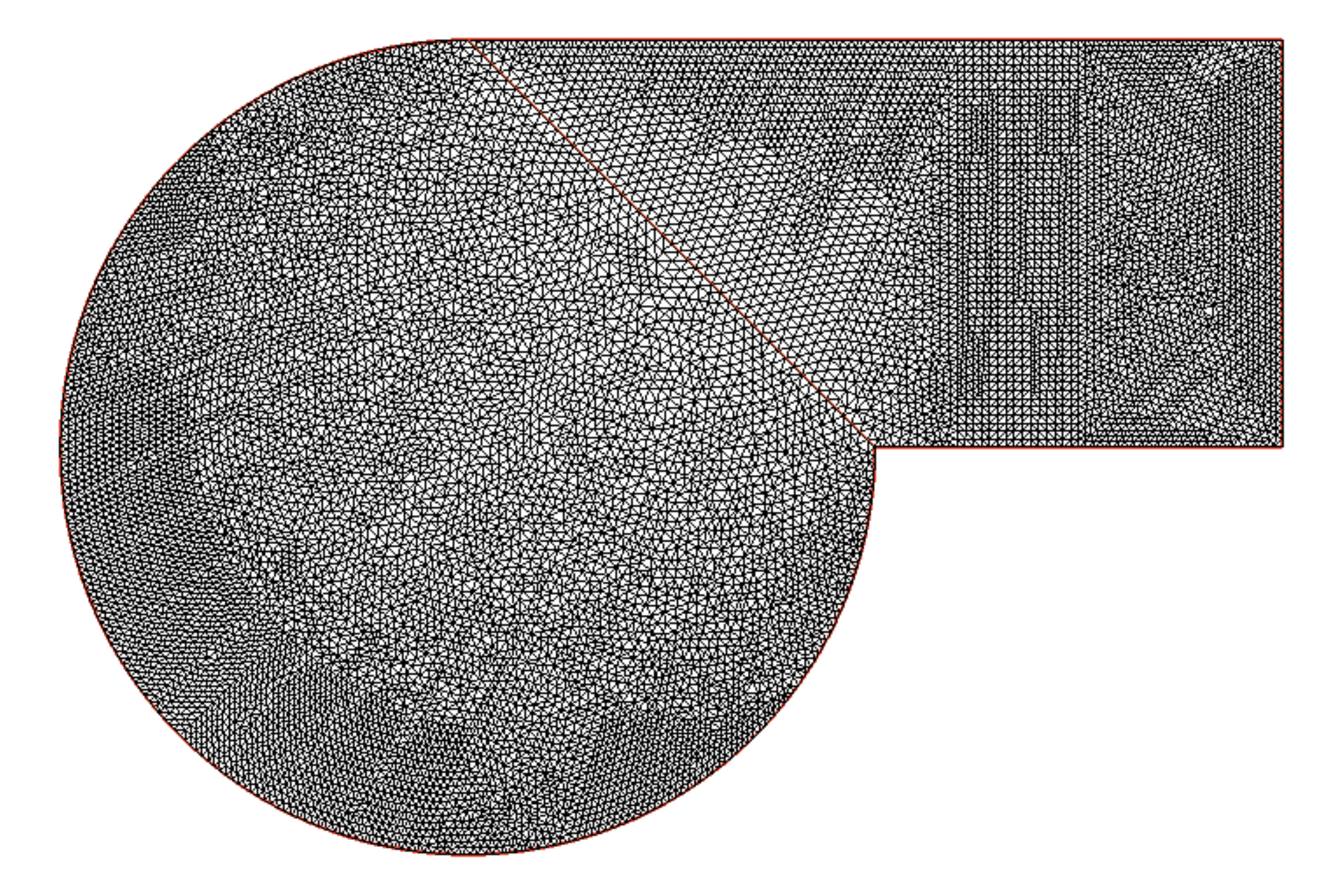}
  \caption{The domain $\Omega$ and its mesh.}
 \end{figure}

We consider over the domain $\Omega\in \R^2$ the Laplace problem: 
\begin{eqnarray}
\label{edp}
&-\Delta \varphi =f,\ \text{in}\ \Omega \\
&f=1+(\alpha \sin(x)+\beta \cos(\gamma \pi y)) \chi_1(x, y)\nonumber
\end{eqnarray}
complemented with homogeneous Dirichlet boundary conditions.
Here $\alpha$, $\beta$ and $\gamma$ are 3 parameters freely chosen  in given intervals in $\R$ that modulate the forcing term on the right hand side. We assume that the forcing term only acts on a part of $\Omega$ named $\Omega_1$ ($\Omega_1 =$ support$(\chi_1)$) and we denote as $\Omega_2$ the remaining part $\Omega_2 = \Omega\setminus\overline\Omega_1$. 

The easy observation is that the solution $\varphi$, depends on the parameters $\alpha, \beta,\gamma$: we plot here one of the possible solutions

\begin{figure}[htbp]\center
   \includegraphics[width=5cm]{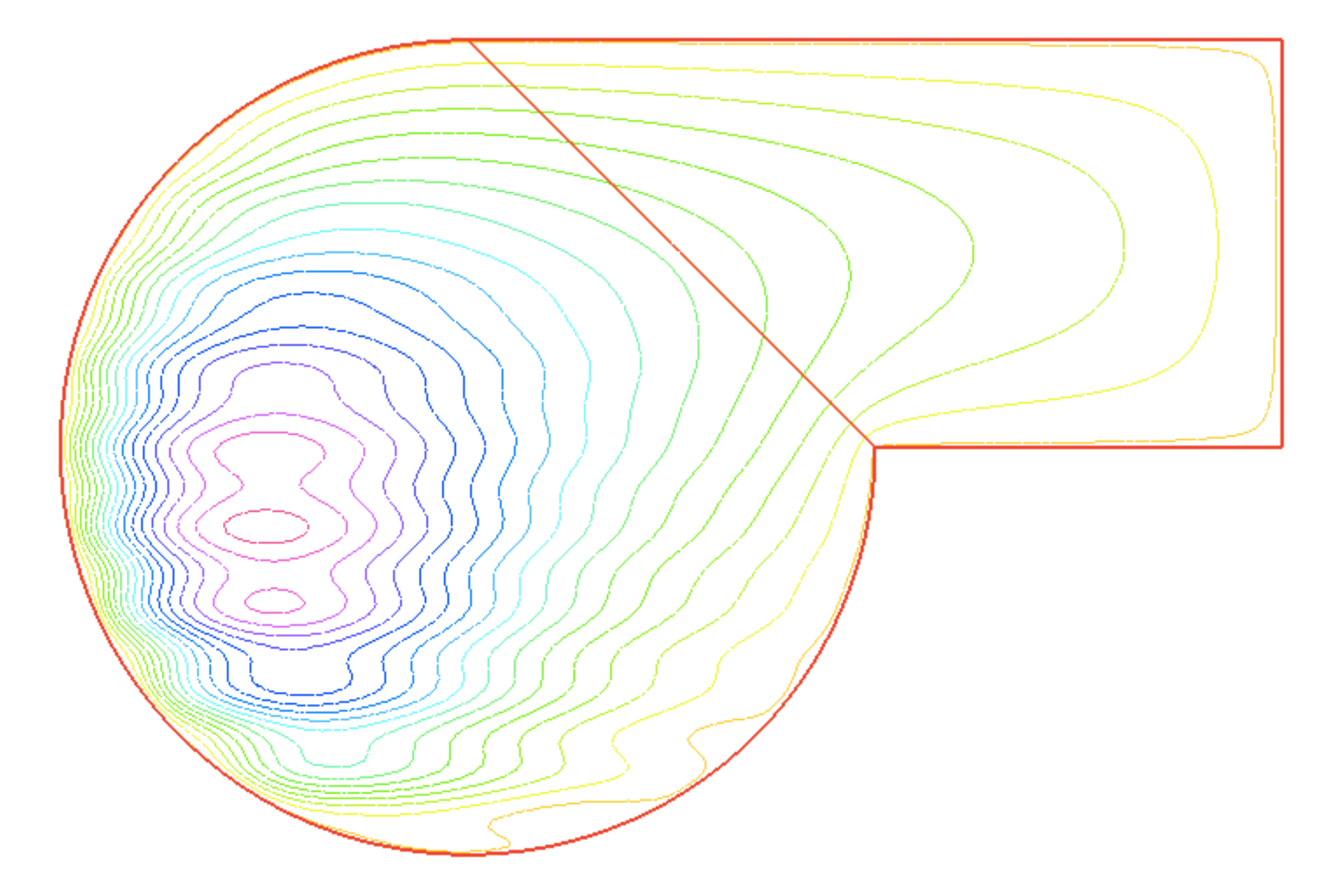}
  \caption{One of the solutions, we note that the effect of the forcing is mainly visible on domain $\Omega_1$ on the left hand side.}
 \end{figure}

We also note that the restriction $\varphi_{\vert \Omega_2}$ to $\Omega_2$ is indirectly dependent on these coefficients and thus is a candidate for building a set (when the parameters vary) of small Kolmogorov width. This can be guessed if we look at the numerical simulations obtained for three representative choices for $\alpha, \beta,\gamma$ 

\begin{figure}[htbp]\center
   \includegraphics[width=4cm]{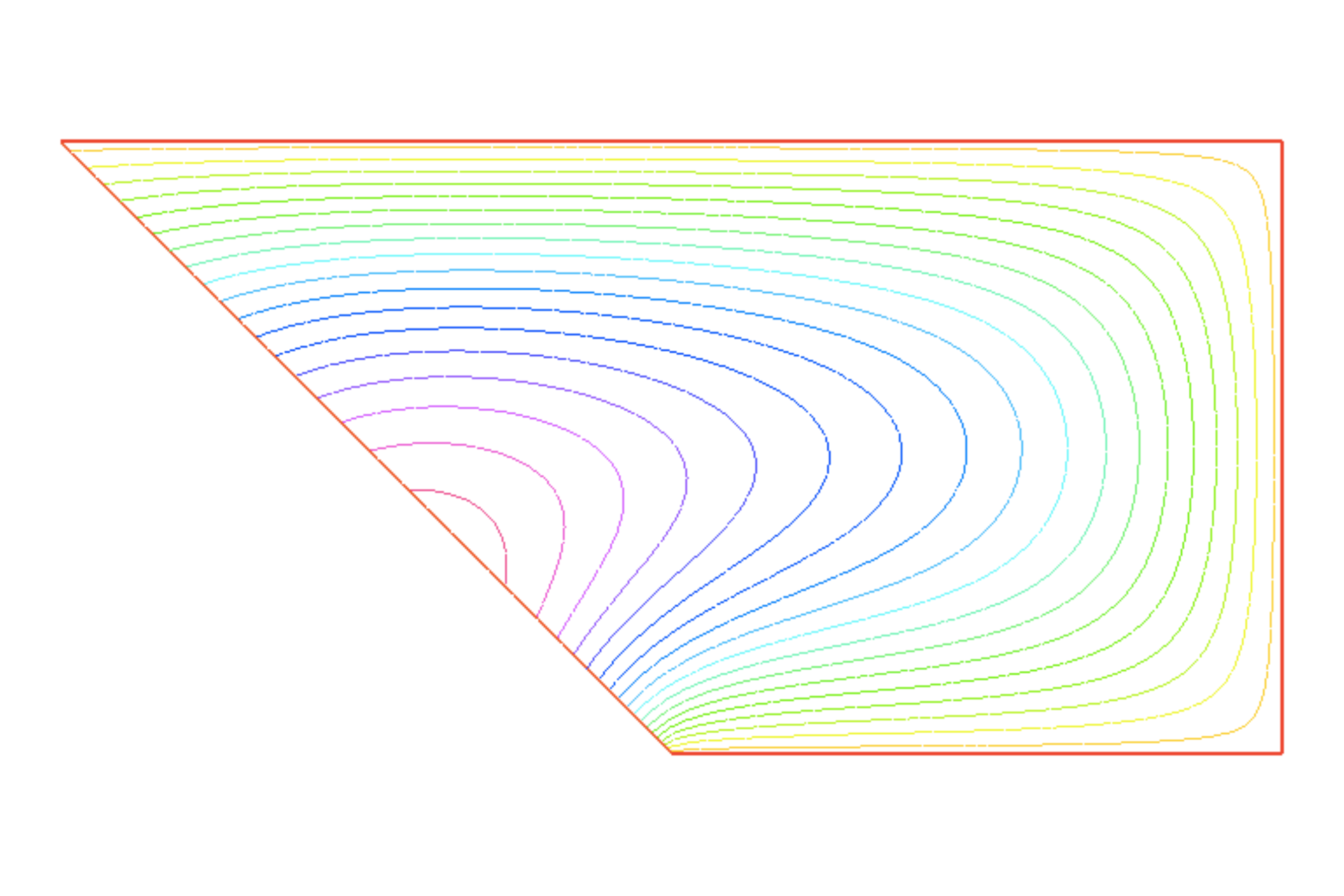}  \includegraphics[width=4cm]{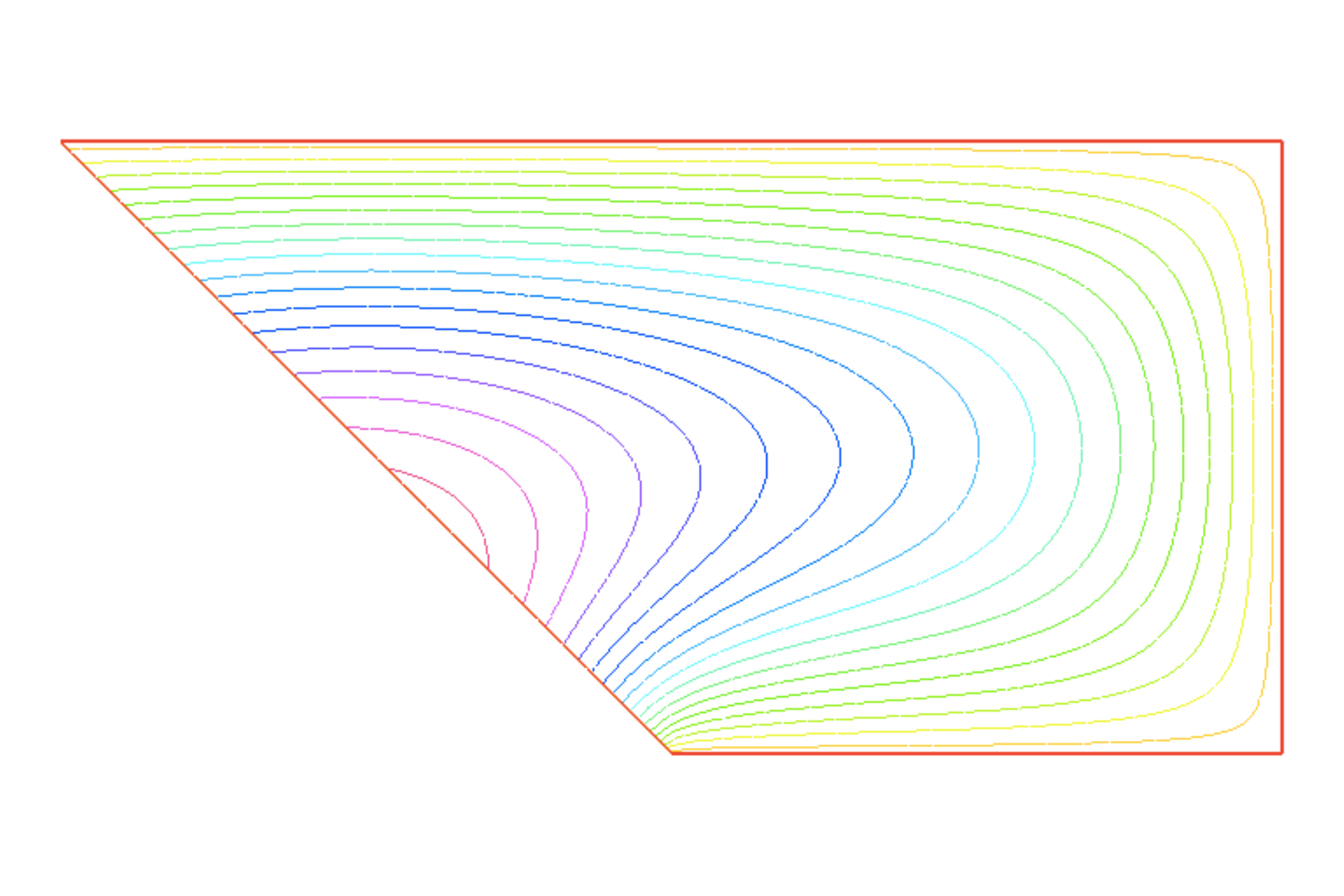}  \includegraphics[width=4cm]{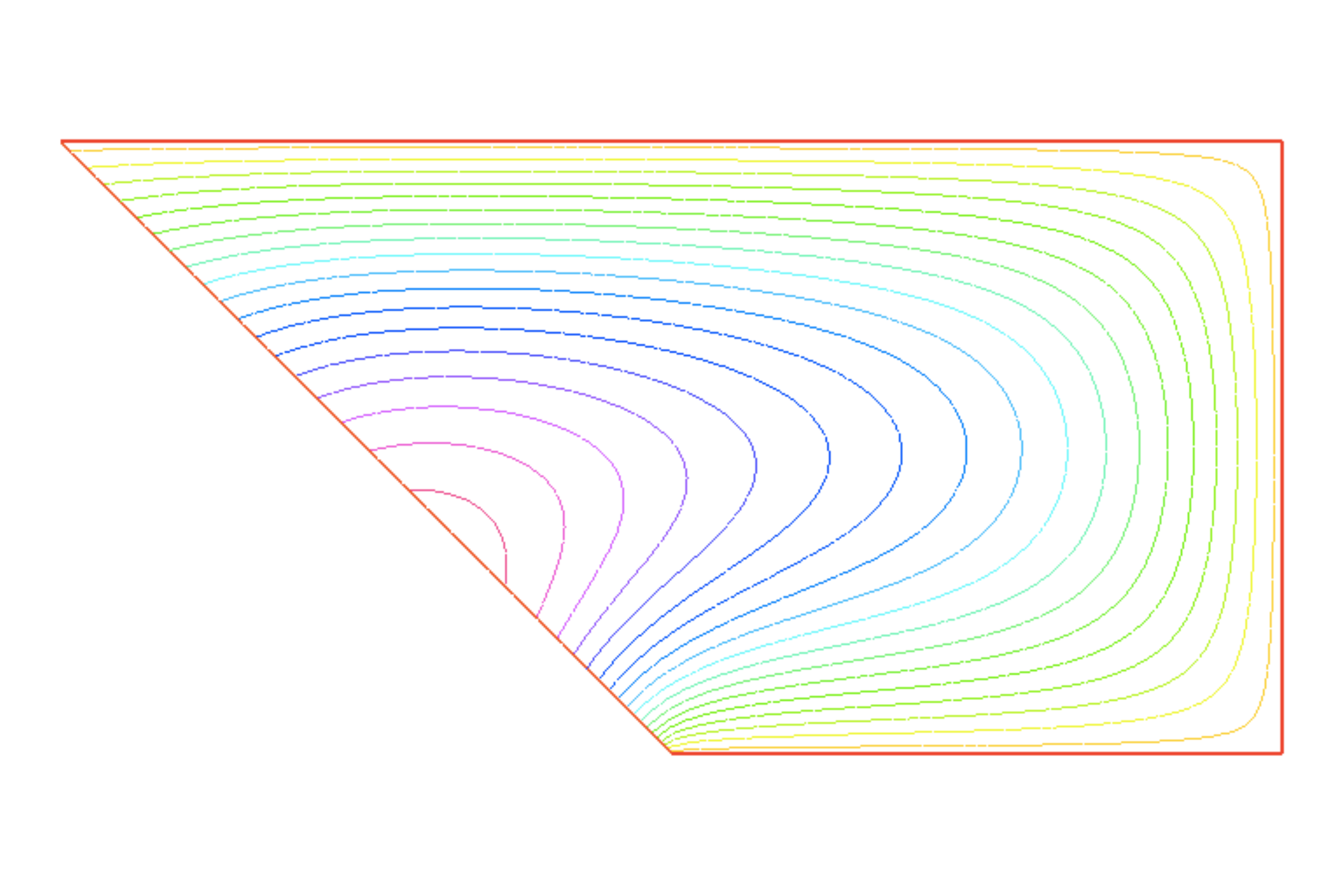}
  \caption{Three generic solutions restricted on the sub-domain  $\Omega_2$.}
 \end{figure}

For the GEIM, we use moments computed from the restriction of the solution $\varphi(\alpha, \beta,\gamma)$ over $\Omega_2$ multiplied by localized functions with small compact support over $\Omega_2$. The reconstructed solutions with the GEIM based on only 5 interpolating functions is $10^{14}$ time better than the reconstructed function with 1 interpolating function illustrating the high order of the reconstruction's convergence.

In the next example, we choose a similar problem but the shape of domain $\Omega_2$ is a further parameter

\begin{figure}[htbp]\center
    \includegraphics[scale=0.35]{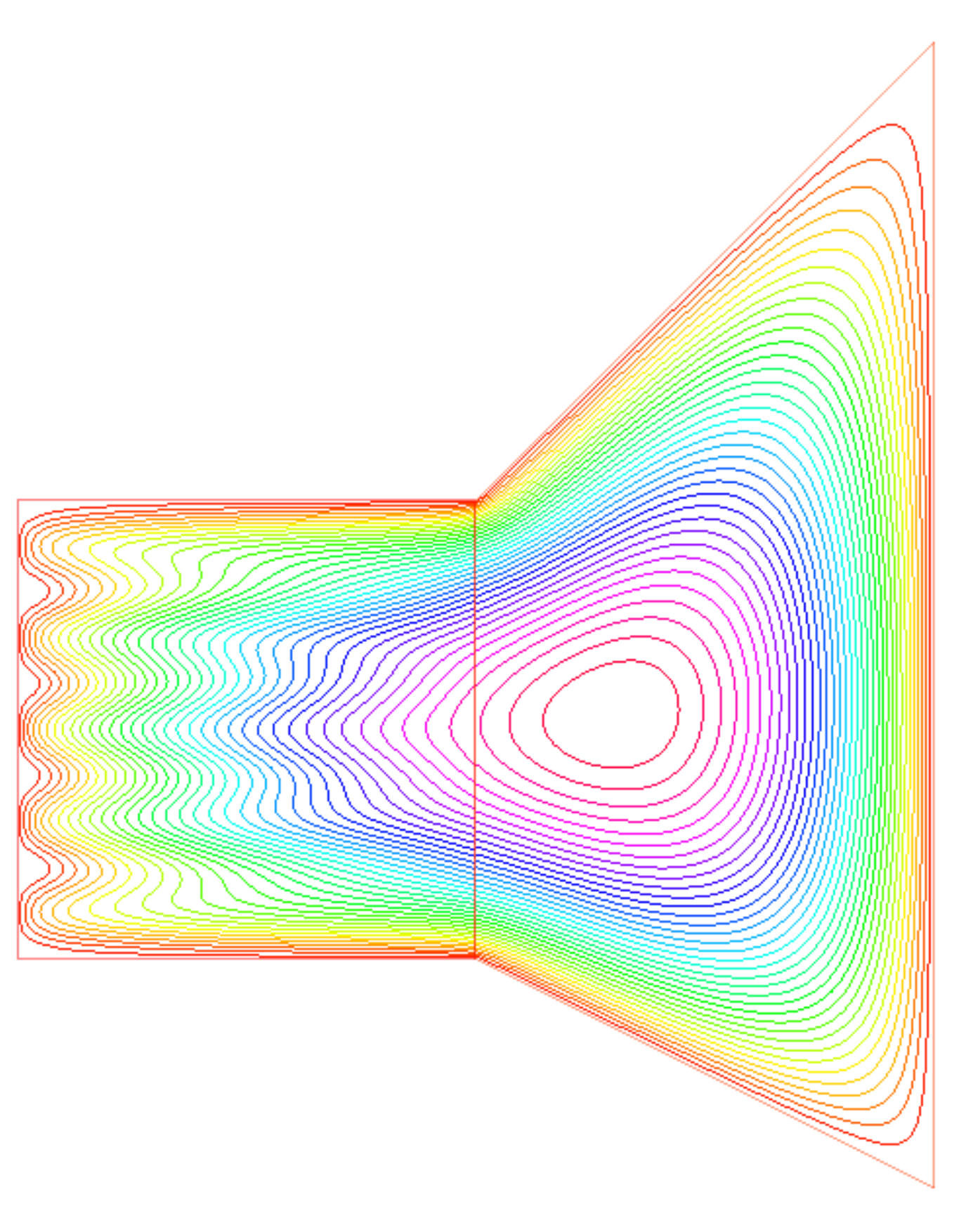}  \includegraphics[scale=0.3]{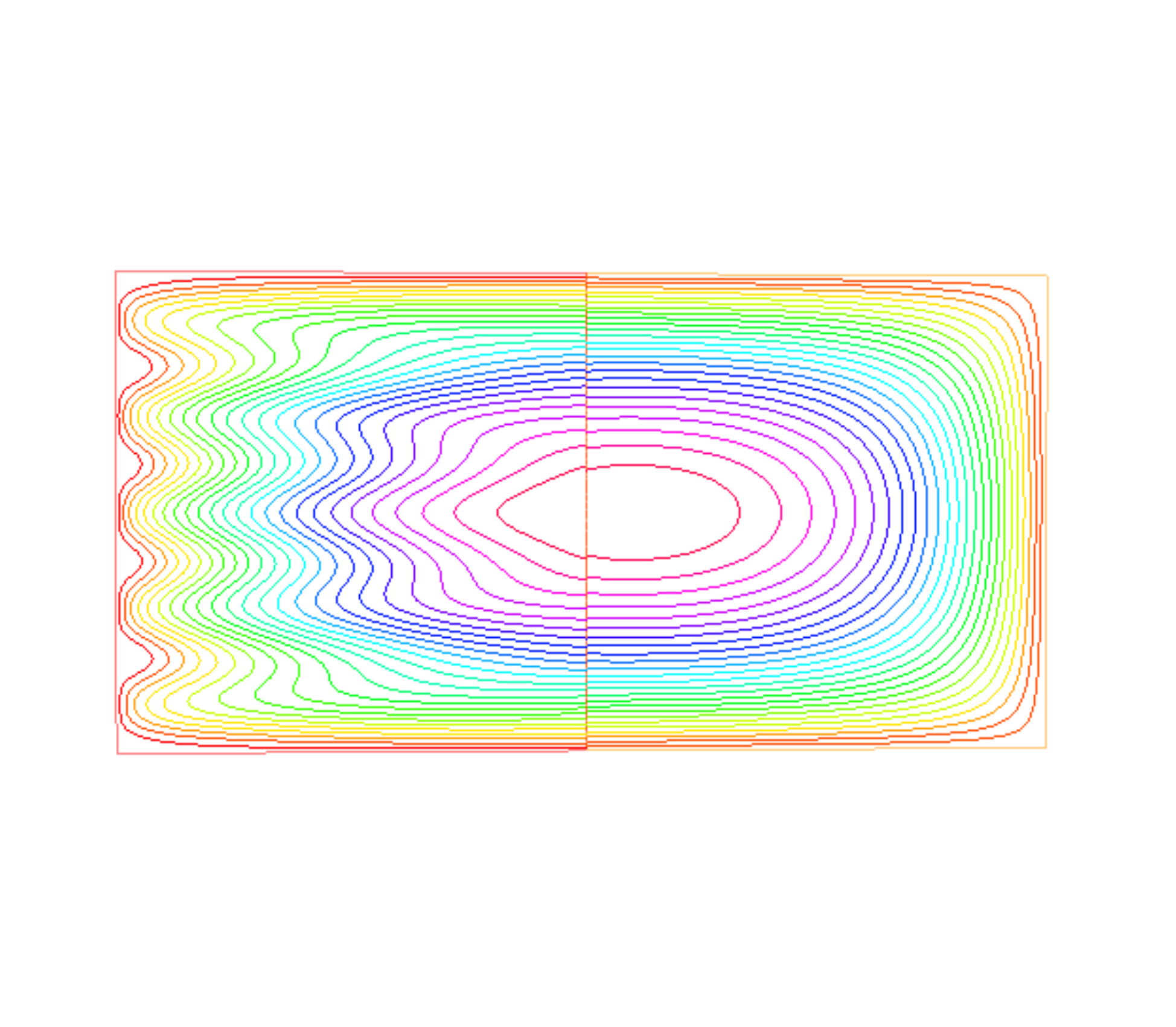}
  \caption{Two generic solutions when shape of the sub-domain  $\Omega_2$ varies.}
 \end{figure}
 
 In order to get an idea of the Kolmogorov width of the set $\{\varphi_{\vert \Omega_2}(\alpha, \beta,\gamma, \Omega_2)$, we perform two Singular Value Decompositions (one in $L^2$, the other in $H^1$) over 256 values (approximated again with Freefem) and plot the decay rate of the eigenvalues ranked in decreasing order: the results are shown on figure 5
 
 \begin{figure}[htbp]\center
    \includegraphics[scale=0.3]{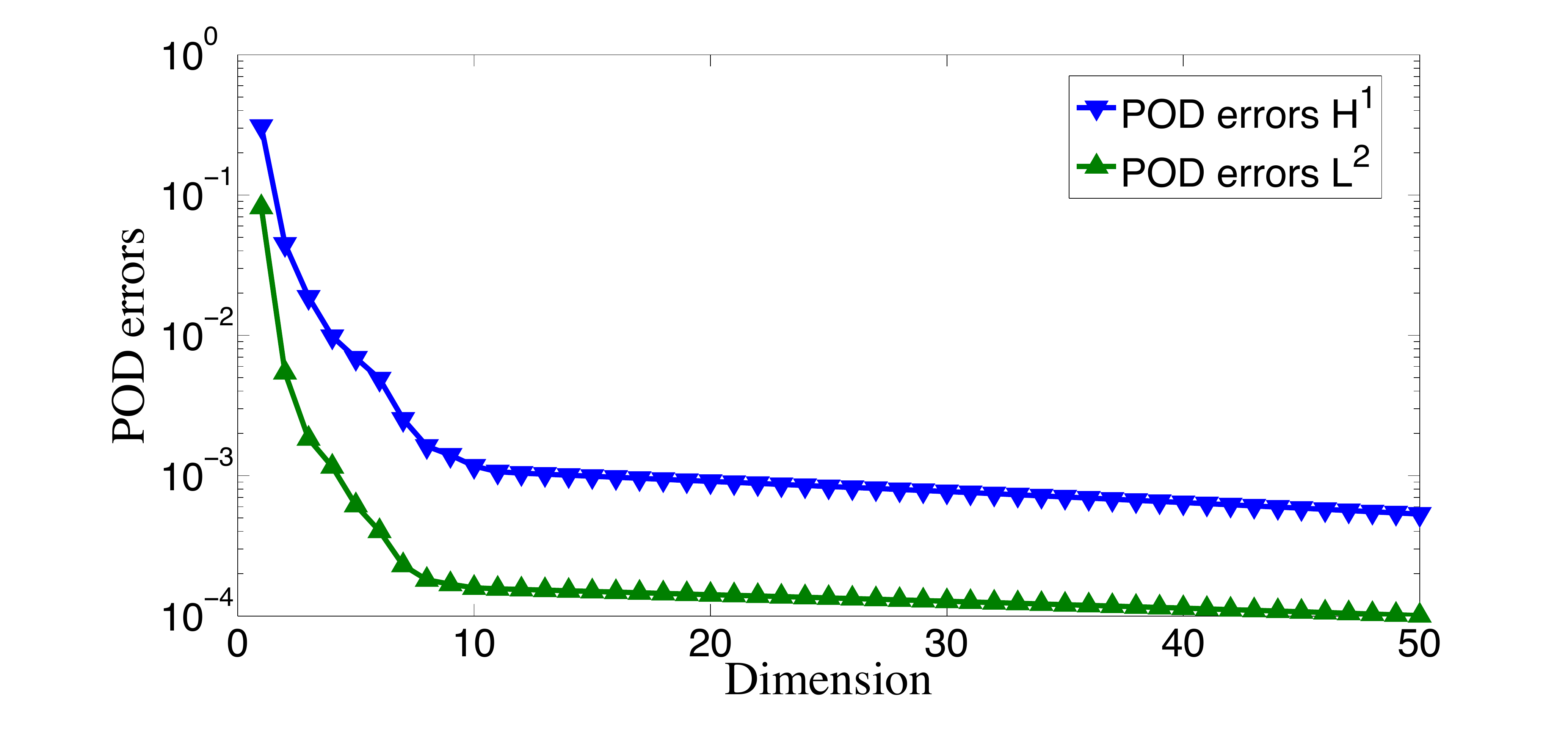}
  \caption{Two SVD ( in $L^2$ and in $H^1$) of the set of solutions over $\Omega_2$.}
 \end{figure}
 We note that after about 9 eigenvalues, the finite element error dominates the decay rate of the true eigenvalues. The GEIM is built  up again with captors represented as local weighted averages over $\Omega_2$. The interpolation error is presented on the next figure (figure 6)

 \begin{figure}[htbp]\center
    \includegraphics[scale=0.35]{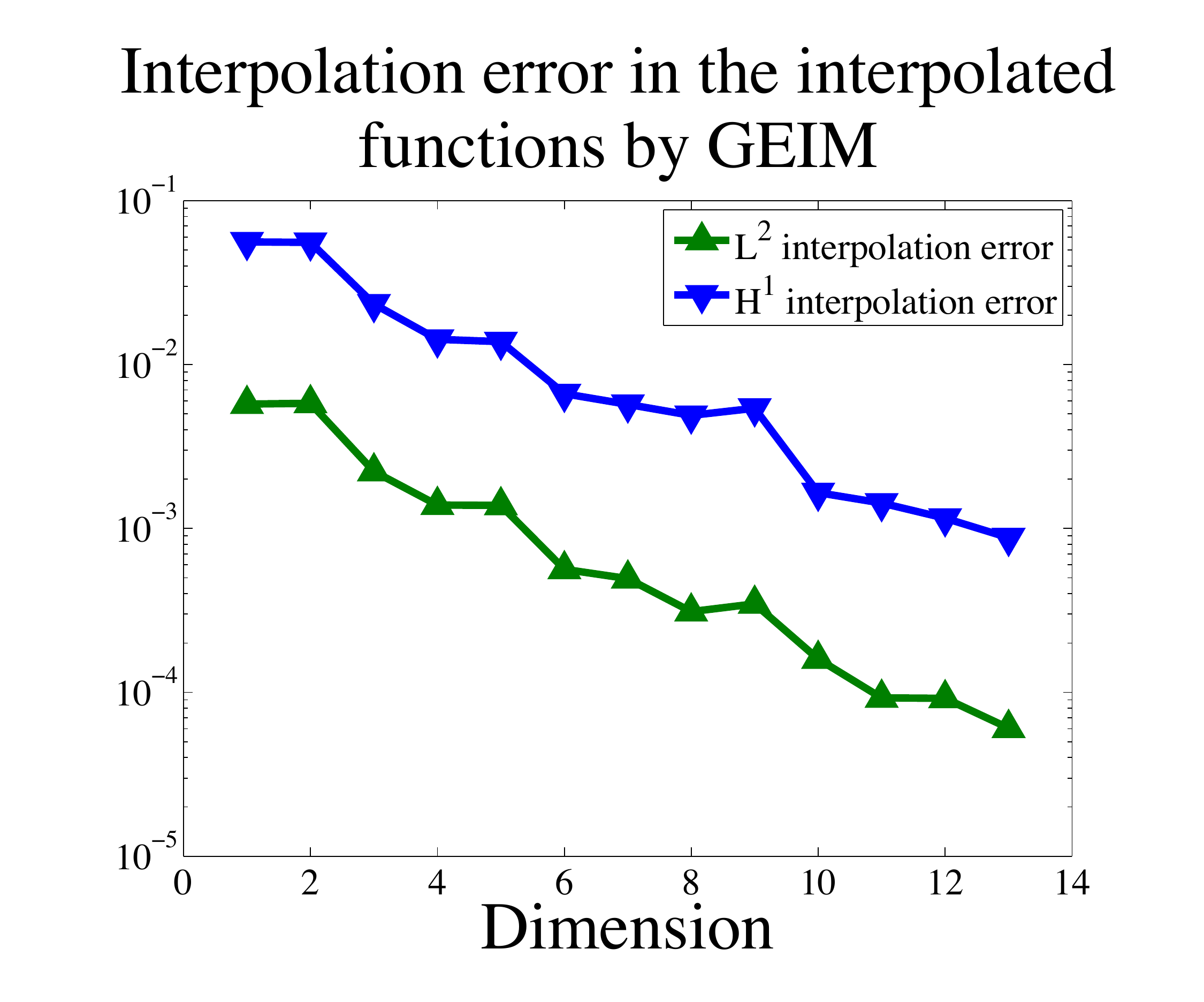}
  \caption{The worse GEIM error with respect to $M$ .}
 \end{figure}
 
 and we note that the decay rate, measured both in $L^2$ and $H^1$ is again quite fast. In order to compare with the best fit represented by the projection, in $L^2$ or in $H^1$, we use the SVD eigenvectors associated with the first $M$ eigenvalues and compare it with ${\cal J}_M$, for various values of $M$. This is represented on figure 7.

 \begin{figure}[htbp]\center
    \includegraphics[scale=0.35]{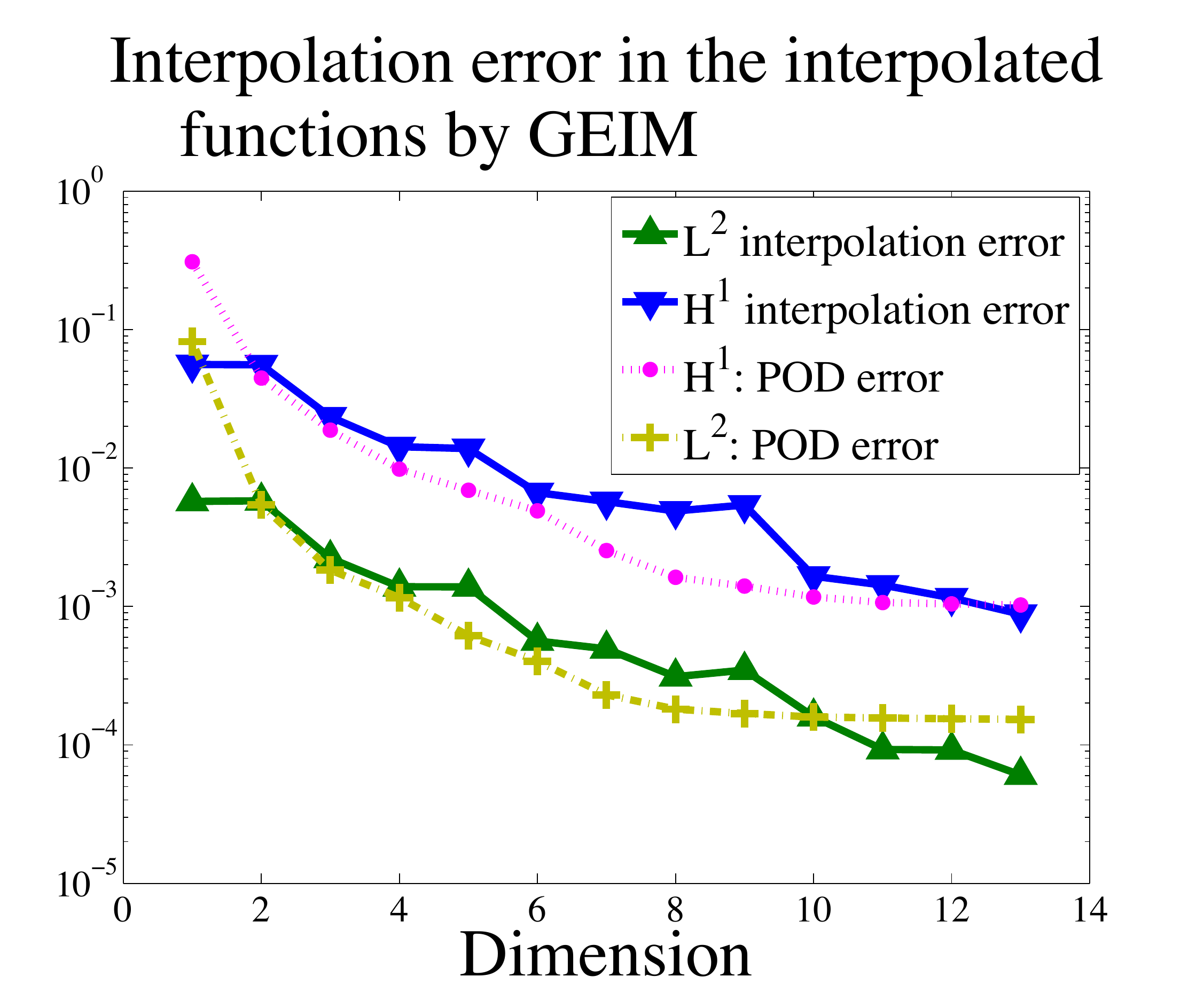}
  \caption{Evolution of the GEIM error versus the best fit error, both in $L^2$ and in $H^1$-norms.}
 \end{figure}
 
 The very good comparison allow to expect that the Lebesgue constant is much better than what is announced in lemma 4. A computational estimation of $\Lambda_M$ has been carried out: \hspace*{1.5cm}$\widetilde{\Lambda_M} = \underset{i\in [1,256]}{\max} \dfrac{\Vert {\cal I}_M[u_i] \Vert_{L^2(\Omega)}}{\Vert u_i \Vert_{L^2(\Omega)}}$

  \begin{figure}[htbp]\center
   \includegraphics[scale=0.3]{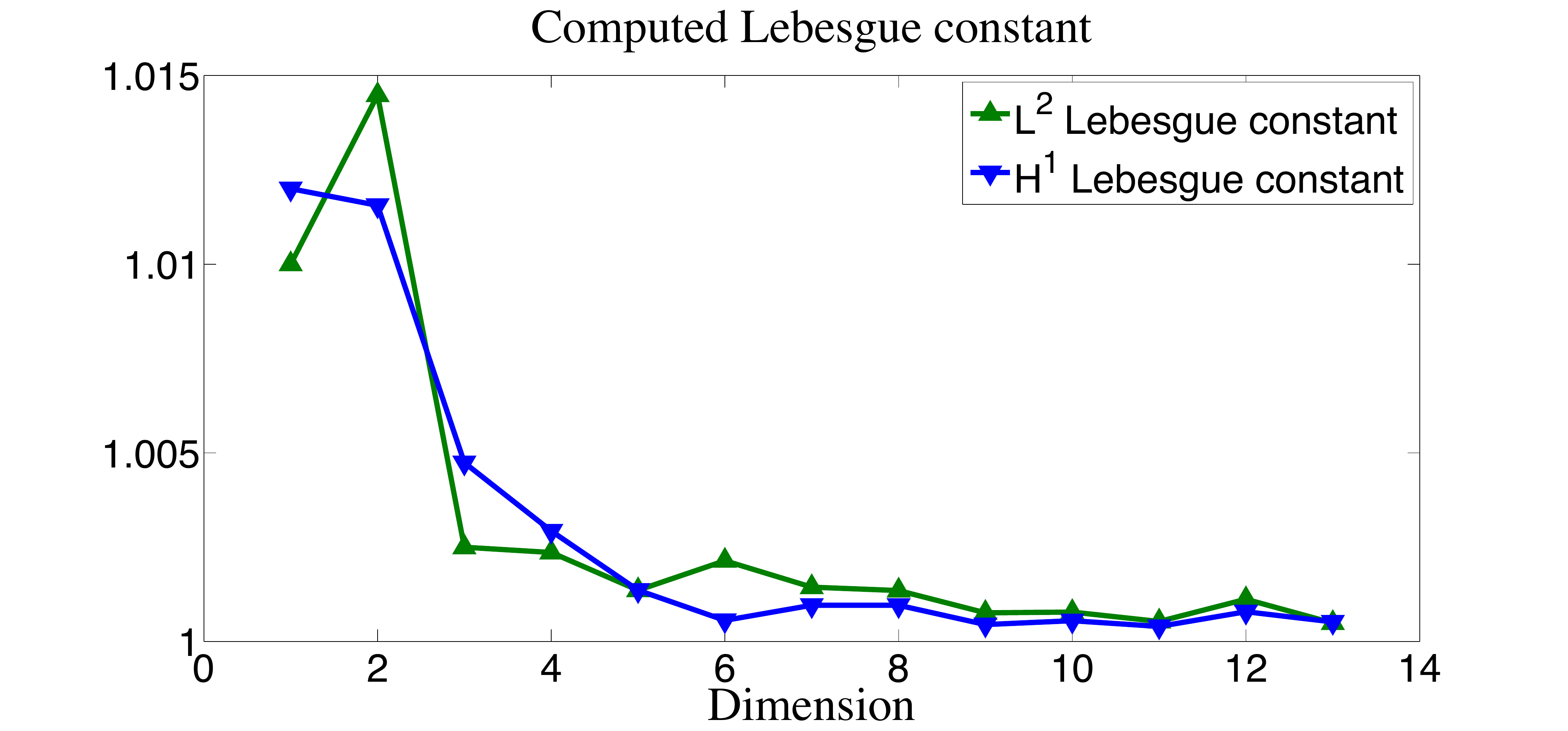}
  \caption{Evolution of the Lebesgue constant, i.e. the norm of the GEIM operator, both in $L^2$ and in $H^1$.}
 \end{figure}

\section{Coupling of deterministic and assimilation methods}

\subsection{The framework}

Imagine that we want to supervise a process in \textbf{real-time} for which we have a parameter dependent PDE. Assume that the computation of the solution over the full domain $\Omega$ is too expensive but we are in a situation where the domain $\Omega$ can be decomposed, as before, into two non overlapping subdomains $\Omega_1$ and $\Omega_2$ and that 
\begin{itemize}
\item $\Omega_1$ is small subdomain but the set of the restriction of the parameter dependent solutions has a 
large Kolmogorov width.
\item $\Omega_2$ is a big subdomain but the set of the restriction of the parameter dependent solutions has a  small Kolmogorov n-width 
\end{itemize}

In addition assume that it is possible to get outputs from sensors based in $\Omega_2$. The GEIM allows to reconstruct accurately the current solution associated to some parameters over $\Omega_2$ and thus is able to build the boundary condition necessary over the interface between $\Omega_1$ and $\Omega_2$ that with the initially given boundary condition over $\partial \Omega$ to be the necessary boundary condition over $\partial\Omega_1$ that complement the original PDE set now over $\Omega_1$ and not $\Omega$ as is illustrated in the next figures.

  \begin{figure}[htbp]\center
    \includegraphics[width=7cm]{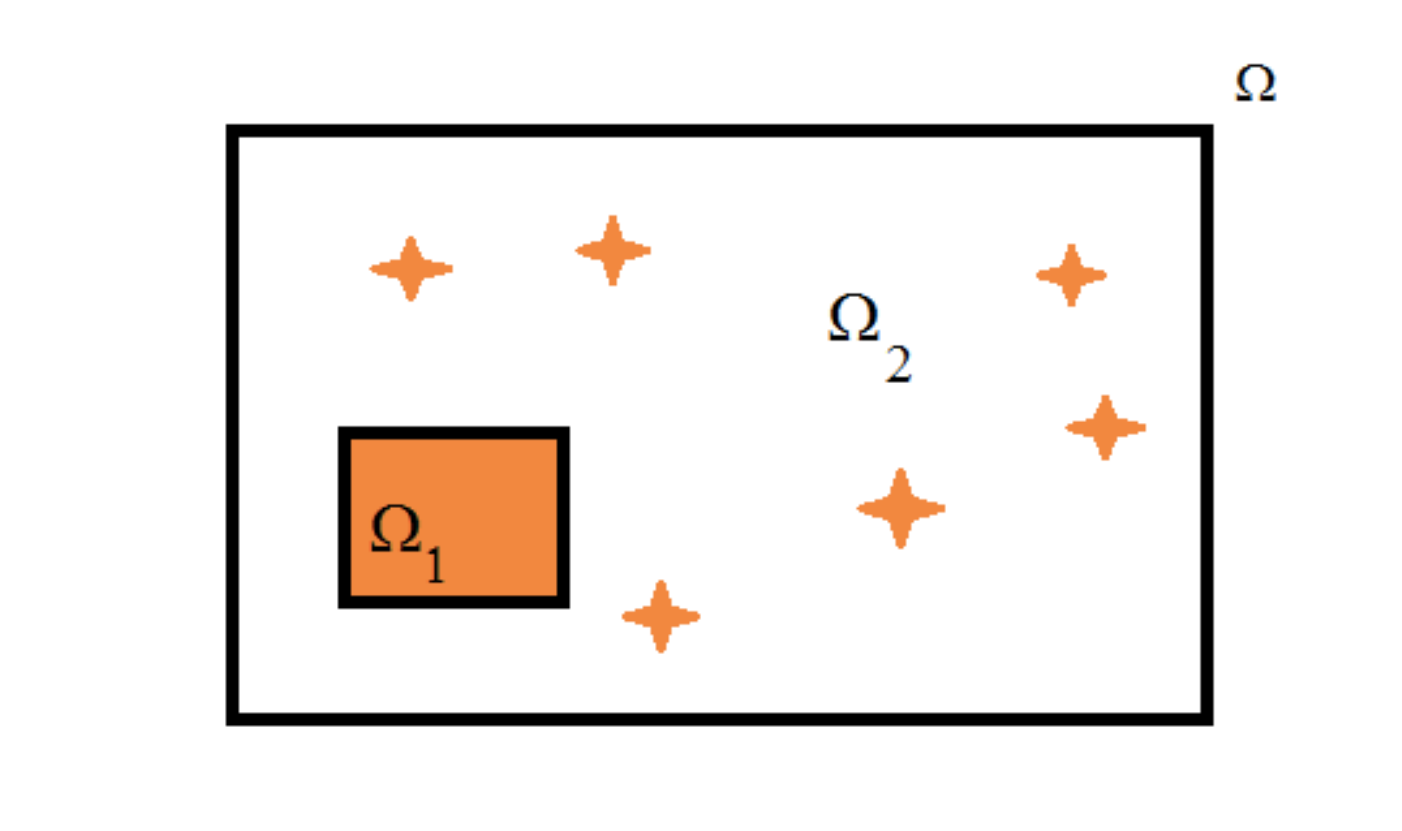}     \includegraphics[width=7cm]{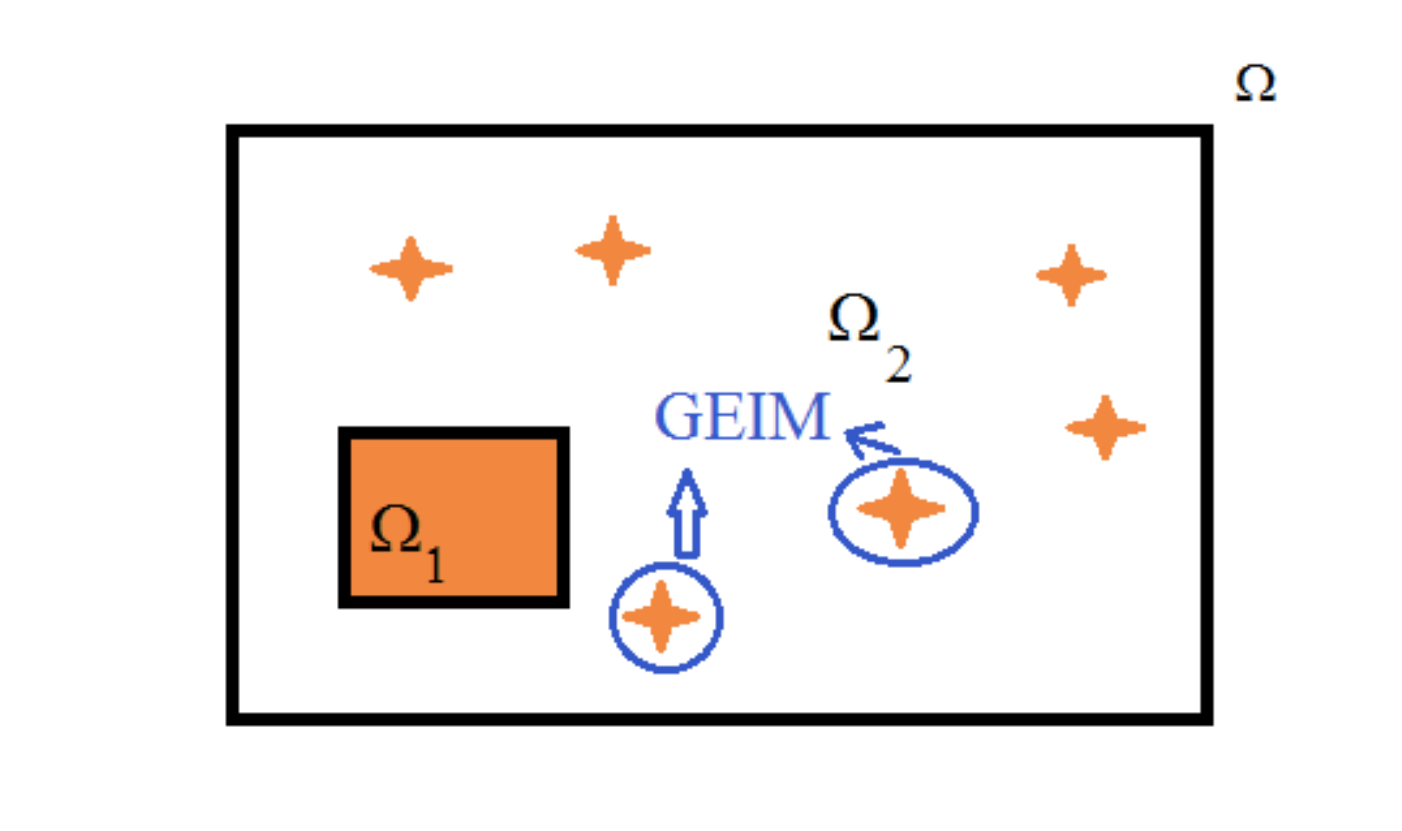}
  \caption{Schematic representation of the reconstruction over $\Omega_2$.}
 \end{figure}

 \begin{figure}[htbp]\center
    \includegraphics[width=7cm]{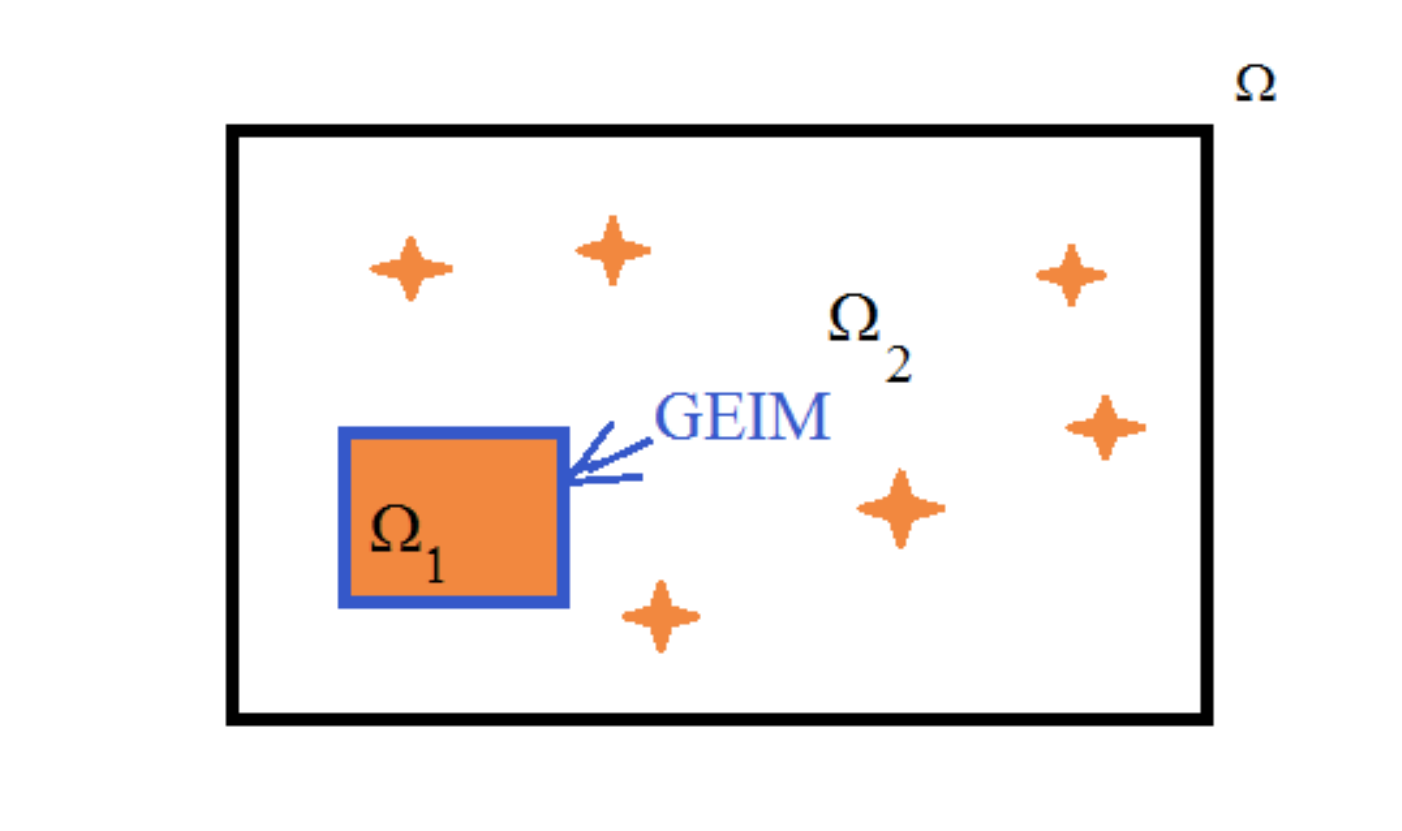}
  \caption{Schematic representation of the recovery over $\Omega_1$ thanks to the knowledge of the interface condition.}
 \end{figure}

\subsection{The combined approach -- numerical results}

We take over the numerical frame of the previous section and go further. We want to apply the GEIM to have a knowledge of the solution $\varphi_{\vert\Omega_2}$ and want to use the trace of the reconstruction on the interface to provide the boundary condition, over $\partial \Omega_1$ to the problem
\begin{eqnarray*}
&-\Delta \varphi =f,\ in\ \Omega_1 \\
&f=1+(\alpha \sin(x)+\beta \cos(\gamma \pi y)) \chi_1(x, y)
\end{eqnarray*}
derived from (\ref{edp}).
 
 The results are presented in figure 11 where both the $H^1$ error on $\varphi_{\vert\Omega_1}$ and $\varphi_{\vert\Omega_2}$ are presented as a function of $M$ being the number of interpolation data that are used to reconstruct $\varphi_{\vert\Omega_2}$. This illustrates that the use of the small Kolmogorov width of the set $\{\varphi_{\vert\Omega_2}\}$ as the parameters vary (including the shape of $\Omega_2$) can help in determining the value of the full $\varphi$ all over $\Omega$.
 
\begin{figure}[htbp]\center
\label{fig11}
    \includegraphics[scale=0.3]{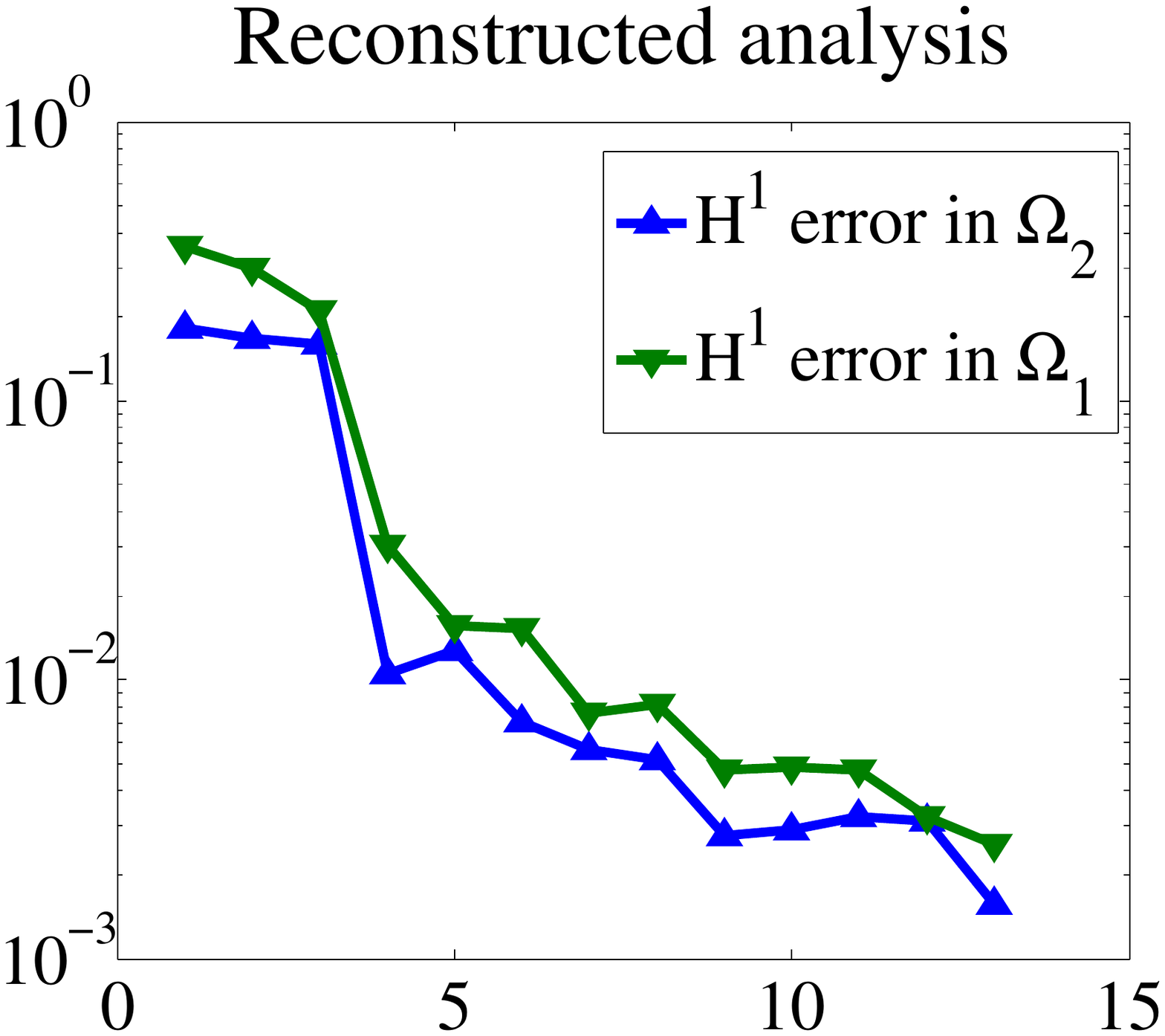}
  \caption{Reconstructed analysis --- error in $H^1$-norm over $\Omega_1$ and $\Omega_2$.}
 \end{figure}

\section{About noisy data}

In practical applications, data are measured with an intrinsic noise due to physical limitations of the sensors. In some sense, the noisy data acquired from the sensors are exact acquisitions from a noisy function that we consider to be a Markovian random field with spacial values locally dependent (on the support of the sensor) and globally independent (from one sensor to the others).
An extension of the previous development needs therefore to be done in order to take this fact under consideration.

Let us assume that all the sensors are subject to the same noise, i.e. provide averages  --- or some moments --- computed, not from $\varphi$, but from a random process $\varphi_\varepsilon \simeq  \mathcal{N}(\varphi, \varepsilon^2)$. The norm of the GEIM operator being equal to $\Lambda_M$ the GEIM-reconstruction forms a random process ${\mathcal J}_M[\varphi_\varepsilon] \simeq \mathcal{N}({\mathcal J}_M[\varphi], \Lambda^2_M \varepsilon^2)$ due to linearity.

Even though the Lebesgue constant seems to be small in practice, we would like to use all the data that are available in order to get a better knowledge of $\varphi$. For the definition of ${\mathcal J}_M$ we indeed only use $M$ data selected out of a large set of all data. 
For this purpose, let us consider that, with some greedy approaches, we have determined $P$ independent series of $M$ different captors $\{ \sigma^{(p)}_1 , \sigma^{(p)}_2 , \dots , \sigma^{(p)}_M  \} ,\ \forall 1\leq p \leq P$. For each of these series, the GEIM  applied to $\varphi$ is noisy and each application provides ${\mathcal J}^p_M[\varphi_\varepsilon] \simeq \mathcal{N}({\mathcal J}^p_M[\varphi], {\Lambda^{p}_M}^2 \varepsilon^2)$. We shall use these $P$ reconstructions by averaging them and expect to improve the variance of the reconstruction.

Let $\lambda^{-1}=\frac{1}{P}\sum\limits^{P}_{p=1} \frac{1}{\Lambda^{p}_N}$. Since the $P$ realizations: $\{{\mathcal J}^p_M[\varphi_\varepsilon]\}_p$ are independent, then the random variable $\overline{{\mathcal J}^P_M}(\varepsilon)  = \frac{\lambda}{P} \sum\limits_{p=1}^{P}\frac{{\mathcal J}^p_M[\varphi_\varepsilon] }{\Lambda^{(p)}_N}$ follows a Gaussian Markov random field of parameters $\mathcal{N}(\mathcal{J}_N (\varphi) , \frac{\epsilon^2 \lambda^2}{P} )$. A realization of this random process could be chosen for an improved estimate of $\mathcal{J}_M (\varphi)$. Indeed, the law of the error  follows $\mathcal{N}(0,\frac{\epsilon^2 \lambda^2}{P})$ and its variance can be less than the size of the initial noise on the captors ($\epsilon$) provided that $\Lambda^{(p)}_N < \sqrt{P} ,\forall 1\leq p \leq P$, which, from the numerical experiments, seems to be the case.

 \section{Conclusions}

We have presented a generalization of the Empirical Interpolation Method, based on ad'hoc interpolating functions and data acquired from sensors of the functions to be represented as those that can arise from data assimilation. We think that the GEIM is already interesting per se as it allows to select in a greedy way the most informative sensors one after the other. It can also propose, in case this is feasible, to build better sensors in order to complement a given family of existing ones and/or detect in which sense some of them are useless because redundant. Finally we also explain how noise on the data can be filtered out.

The coupled use of GEIM with reduced domain simulation is also proposed based on  domain decomposition technique leading to a small portion where numerical simulation is performed and a larger one based on data assimilation.
 
We think that the frame presented here can be used as an alternative to classical Bayesian or frequentistic statistic where the knowledge developed on the side for building mathematical models and their simulations can be fully used for data mining (we refer also to \cite{patera} and \cite{rozza} for recent contributions in this direction).

 \section*{Acknowledgements}

This work was supported in part by the joint research program MANON between CEA-Saclay and University Pierre et Marie Curie-Paris 6. We want to thank G. Biot from LSTA and G. Pag\`es from LPMA for constructive discussions on the subject.


\begin{thebibliography}{99}


 
\bibitem{barrault04:_empir_inter_method}
M.~Barrault, N.~C. Nguyen, Y.~Maday, and A.~T. Patera.
\newblock An ``empirical interpolation'' method: Application to efficient
  reduced-basis discretization of partial differential equations.
\newblock {\em C. R. Acad. Sci. Paris, S{\'e}rie I.}, 339:667--672, 2004.


\bibitem{maday1} Chakir R.; Maday Y. A two-grid finite-element/reduced basis scheme for the approximation of the solution of parametric dependent PDE.
Comptes Rendus Mathematiques {\bf 347}, 435-440   DOI: 10.1016/j.crma.2009.02.019.

\bibitem{freefem} http://www.freefem.org

\bibitem{m2an_magic}
M.~A. Grepl, Y.~Maday, N.~C. Nguyen, and A.~T. Patera.
\newblock Efficient reduced-basis treatment of nonaffine and nonlinear partial
  differential equations.
\newblock {\em M2AN (Math. Model. Numer. Anal.)}, 2007, 

\bibitem{kolmo} A. Kolmogoroff, \"{U}ber die beste {A}nn\"aherung von {F}unktionen einer
              gegebenen {F}unktionenklasse, Annals of Mathematics. Second Series, {\bf 37}, 107--110, 1936.
 

\bibitem{magic}
Y. Maday,  and N. C. Nguyen and A. T. Patera  and
             G. S. H. Pau,
              \newblock {A general multipurpose interpolation procedure: the magic
              points},
              Commun. Pure Appl. Anal., {\bf 8}1, (2009), 383--404 


\bibitem{nguyen04:_handb_mater_model}
N.~C. Nguyen, K.~Veroy, and A.~T. Patera.
\newblock Certified real-time solution of parametrized partial differential
  equations.
\newblock In S.~Yip, editor, {\em Handbook of Materials Modeling}, pages
  1523--1558. Springer, 2005.
  
\bibitem{patera}  A.T.  Patera and EM R\o nquist, Regression on Parametric Manifolds: Estimation of Spatial Fields, Functional Outputs, and Parameters from Noisy Data. CR Acad Sci Paris, Series I, 350(9-10):543-547, 2012.
  
  
  
\bibitem{PRbook} 
 {\sc  A.T. Patera and G. Rozza}, {\em  Reduced Basis Approximation and A Posteriori Error Estimation
for Parametrized Partial Differential Equations}, {\sl Version 1.0, Copyright MIT 2006,
to appear in (tentative rubric) MIT Pappalardo Graduate Monographs in Mechanical
Engineering.}




  \bibitem{prud'homme02:_reliab_real_time_solut_param}
C.~Prud'homme, D.~Rovas, K.~Veroy, Y.~Maday, A.~T. Patera, and G.~Turinici.
\newblock Reliable real-time solution of parametrized partial differential
  equations: Reduced-basis output bound methods.
\newblock {\em Journal of Fluids Engineering}, 124(1):70--80, March 2002.

 \bibitem{rozza} G. Rozza,   Andrea M.,   and F. Negri, Reduction strategies for PDE-constrained optimization problems in haemodynamics, MATHICSE Technical Report, Nr. 26.2012 July 2012 


\bibitem{veroy03:_poster_error_bound_reduc_basis}
K.~Veroy, C.~Prud'homme, D.~V. Rovas, and A.~T. Patera.
\newblock {\em A posteriori\/} error bounds for reduced-basis approximation of
  parametrized noncoercive and nonlinear elliptic partial differential
  equations ({AIAA P}aper 2003-3847).
\newblock In {\em Proceedings of the 16th AIAA Computational Fluid Dynamics
  Conference}, June 2003.





\end{thebibliography}
\end{document}